\documentclass[12pt]{amsart}

\usepackage[T1]{fontenc}
\usepackage{amsmath,amssymb,amsthm,mathtools}
\usepackage{enumitem}
\usepackage[plainpages=false,hypertexnames=false,colorlinks=true,linkcolor=blue,citecolor=blue,urlcolor=blue]{hyperref}
\usepackage[
  a4paper,
  width=6.05in,
  height=8.35in,
  hmarginratio=1:1,
  vmarginratio=1:1
]{geometry}

\numberwithin{equation}{section}

\newtheorem{prop}{Proposition}[section]
\newtheorem{thm}[prop]{Theorem}
\newtheorem{lem}[prop]{Lemma}
\newtheorem{cor}[prop]{Corollary}
\theoremstyle{remark}
\newtheorem{rem}[prop]{Remark}

\DeclareMathOperator{\tr}{tr}
\DeclareMathOperator{\osc}{osc}
\DeclareMathOperator{\cof}{cof}
\newcommand{\R}{\mathbb R}
\newcommand{\T}{\mathbb T}

\allowdisplaybreaks
\title[Coupled-drift Monge--Amp\`ere equations]
{Liouville-type theorems for coupled-drift
Monge--Amp\`ere equations}

\author{Ling Wang}
\address{Department of Decision Sciences and BIDSA, Bocconi University, Milano, Italy}
\email{ling.wang@unibocconi.it}


\begin{document}

\begin{abstract}
In this paper, we study entire solutions and periodic correctors for the coupled-drift
Monge--Amp\`ere equation
\[
        \det D^2u
        =
        \exp\{-a\cdot Du+b\cdot x+V(x)-c_0\},
        \quad
        D^2u>0.
\]
For \(V\equiv0\), we obtain a sharp classification of the whole-space
solvability regimes: all entire smooth strictly convex solutions are quadratic
when \(a=b=0\); no such solution exists when \(a\neq0\) and
\(a\cdot b\le0\); and non-quadratic entire solutions exist when
\(a=0\) and \(b\neq0\), or when \(a\cdot b>0\).  For the null case
\(a\neq0\), \(a\cdot b=0\), we give a scalar maximum-principle argument in
every dimension \(n\ge2\).

For periodic \(V\), we prove existence and uniqueness of the normalized pair
\((\psi_A,c_A)\) solving the drifted cell problem
\[
        \det(A+D^2\psi)
        =
        \exp\{-a\cdot D\psi+V-c_A\},
        \quad
        A+D^2\psi>0
        \quad\text{on }\mathbb T^n.
\]
We also prove that any asymptotically quadratic entire solution must satisfy
\(b=Aa\).  If its remainder is bounded, then the solution is the corresponding
quadratic-periodic corrector up to an additive constant.
\end{abstract}

\maketitle

\section{Introduction}

In this paper, we study entire smooth strictly convex solutions of the
Monge--Amp\`ere equation
\begin{equation}\label{eq:intro-main}
        \det D^2u
        =
        \exp\{-a\cdot Du+b\cdot x+V(x)-c_0\},
        \quad
        D^2u>0,
        \quad
        x\in\mathbb R^n,
\end{equation}
where \(a,b\in\mathbb R^n\), \(c_0\in\mathbb R\), and \(V\) is either zero or a
smooth periodic function on \(\mathbb T^n\).  When \(V\equiv0\), the Legendre
transform interchanges the two drift terms.  Indeed, if \(y=Du(x)\) and \(f=u^*\), then on the gradient image
\(Du(\mathbb R^n)\),
\[
        \det D^2f
        =
        \exp\{a\cdot y-b\cdot Df+c_0\}.
\]
For the whole-space problem, the sign of \(a\cdot b\) determines the
solvability.  For periodic correctors, the corresponding condition is
\(b=Aa\), where \(A\) is the quadratic part of the solution.

The starting point is the classical rigidity theory for the real
Monge--Amp\`ere equation.  The classical J\"orgens--Calabi--Pogorelov theorem asserts
that every entire smooth strictly convex solution of
\[
        \det D^2u=\mathrm{constant}
        \quad\text{in }\mathbb R^n
\]
is quadratic; see \cite{Jorgens,Calabi,Pogorelov72}.  Cheng and Yau
\cite{ChengYau86} gave an analytic proof in their work on complete affine
hypersurfaces, and Caffarelli and Li \cite{CaffarelliLi03} extended the theory
to viscosity solutions and exterior-domain problems.  Existence and qualitative
questions for entire equations with nonconstant prescribed densities were also
studied by Chou and Wang \cite{ChouWang96}.  For systematic treatments of the real Monge--Amp\`ere equation and its
linearized counterpart, see
\cite{Gutierrez16,Figalli17,Le24}.  

A different but closely related rigidity problem arises in affine
K\"ahler--Ricci flat geometry.  After a Legendre transform, the potential
equation takes the form
\[
        \det D^2u
        =
        \exp\{-d\cdot Du-d_0\},
        \quad x\in\mathbb R^n,
\]
where \(d\in\mathbb R^n\) and \(d_0\in\mathbb R\) are constants.  Li and Xu
\cite{LiXu09} proved that no entire smooth strictly convex solution exists
when \(d\neq0\); equivalently, the only entire smooth strictly convex solutions
in this class occur in the constant-density case.  Xu and Zhu
\cite{XuZhu16} later gave a shorter proof of this result.  These results show
that a linear gradient drift is not an innocuous perturbation of the
constant-density equation.  For the broader affine Bernstein problem and
related Monge--Amp\`ere equations, see
\cite{LiXuSimonJia10,TrudingerWang00,TrudingerWang02,TrudingerWang05,
TrudingerWang08}.

Our first goal is to determine what happens when the gradient drift is coupled
to a spatial drift.  In the constant-medium case of \eqref{eq:intro-main}, namely
\begin{equation}\label{eq:linear-drift}
        \det D^2u
        =
        \exp\{-a\cdot Du+b\cdot x-c_0\},
        \quad
        D^2u>0,
        \quad
        x\in\mathbb R^n.
\end{equation}
This equation also appears in the study of Lagrangian translating solitons in
the pseudo-Euclidean space \(\mathbb R^{2n}_n\); see
Joyce--Lee--Tsui \cite{JoyceLeeTsui10} for related translating solitons in
the Euclidean setting.  Xu and Huang
\cite{XuHuang13} derived the equation for spacelike Lagrangian translating
gradient graphs and proved rigidity under a completeness assumption.  Further
results were obtained in \cite{HuangXu15,XuZhu15,ChenQiu16,XuLiu19}.  In
particular, Xu and Zhu \cite{XuZhu15} classified the one-dimensional entire
solutions and obtained higher-dimensional rigidity under an asymptotic decay
condition on the Hessian.

In the translating-soliton interpretation, the sign of \(a\cdot b\)
corresponds to the causal type of the translating vector, up to the convention
for the null metric.  Wu and Xu \cite{WuXu21} proved nonexistence when
\(a\cdot b<0\).  In dimension two, Song--Wu--Xu \cite{SongWuXu26} proved
nonexistence in the lightlike case when both drift vectors are nonzero.  Our
first result gives the complete solvability alternatives for
\eqref{eq:linear-drift}.

\begin{thm}
\label{thm:linear-dichotomy}
Let \(n\ge1\), \(a,b\in\mathbb R^n\), and \(c_0\in\mathbb R\).  Consider
\eqref{eq:linear-drift}.  Then the following statements hold.
\begin{enumerate}
\item[\textnormal{(i)}]
If \(a=b=0\), then every entire smooth strictly convex solution of
\eqref{eq:linear-drift} is a quadratic polynomial.

\item[\textnormal{(ii)}]
If \(a\neq0\) and \(a\cdot b\le0\), then \eqref{eq:linear-drift} admits no
entire smooth strictly convex solution.

\item[\textnormal{(iii)}]
If either \(a=0\) and \(b\neq0\), or \(a\cdot b>0\), then
\eqref{eq:linear-drift} admits non-quadratic entire smooth strictly convex
solutions.
\end{enumerate}
\end{thm}

The new rigidity regime in Theorem~\ref{thm:linear-dichotomy} is the
null-coupling case
\[
        a\neq0,
        \quad
        a\cdot b=0.
\]
Within the class of arbitrary entire smooth strictly convex solutions,
Song--Wu--Xu \cite{SongWuXu26} proved the mixed-null nonexistence result in
dimension two.  Here we remove the dimension restriction and simultaneously
include the pure gradient-drift case \(b=0\).

Our proof exploits a scalar structure specific to the coupled-drift equation.
A sublevel-set maximum principle first gives a global oscillation bound for
\(w=a\cdot Du\).
Writing
\[
        P=a^\top D^2u\,a,
\]
we show that, in the null-coupling case, the full drifted linearized operator
makes \(w\) harmonic and yields an exact nonnegative square identity for
\(P\).  A maximum-principle argument applied to a negative power of \(P\),
together with a properness penalization, then gives the contradiction.  This
argument applies in every dimension \(n\ge2\), while the case \(n=1\) follows
from an elementary ODE calculation.  Neither part requires completeness of
the Hessian metric or uniform ellipticity.

We next turn to periodic media.  Periodic Monge--Amp\`ere equations also
arise naturally on compact special affine and Hessian manifolds.  The
foundational solvability theory in this setting goes back to Cheng and Yau
\cite{ChengYau82}.  Li \cite{Li90} developed existence and uniqueness results
for classes of fully nonlinear Monge--Amp\`ere-type equations on compact
manifolds, including further results on flat tori, and Caffarelli and
Viaclovsky \cite{CaffarelliViaclovsky01} established regularity results on
Hessian manifolds.

For the classical undrifted equation, periodicity of the Monge--Amp\`ere
density leads to a quadratic-periodic decomposition.  Caffarelli and Li
\cite{CaffarelliLi04} proved that, for a smooth positive periodic density,
every entire smooth strictly convex solution has the form
\[
        u(x)
        =
        \frac12 x^\top A x+\ell\cdot x+\psi(x),
\]
where \(A=A^\top>0\), \(\ell\in\mathbb R^n\), and \(\psi\) is periodic.  Li and
Lu \cite{LiLu22} extended this Liouville theorem to periodic densities
satisfying \(\log f\in L^\infty\).  Related existence and Liouville-type
results for quadratic-growth solutions of more general fully nonlinear elliptic
equations with periodic data, including applications to the Monge--Amp\`ere
equation, were obtained by Li and Liang \cite{LiLiang26}.  Jin--Li--Tran--Tu
\cite{JinLiTranTu25} treated periodic Monge--Amp\`ere measures which are
allowed to be degenerate or singular.  For
periodic homogenization, see
\cite{BensoussanLionsPapanicolaou78,Evans89,Evans92}.

For a prescribed matrix \(A=A^\top>0\), the natural cell problem associated
with \eqref{eq:intro-main} is
\begin{equation}\label{eq:intro-cell}
        \det(A+D^2\psi)
        =
        \exp\{-a\cdot D\psi+V-c\},
        \quad
        A+D^2\psi>0
        \quad\text{on }\mathbb T^n,
\end{equation}
where the constant \(c\) is also unknown.  The second result of this paper is
the following.

\begin{thm}
\label{thm:cell}
For every \(A=A^\top>0\), \(a\in\mathbb R^n\), and
\(V\in C^\infty(\mathbb T^n)\), there exists a unique pair
\((\psi_A,c_A)\) such that
\[
        \psi_A\in C^\infty(\mathbb T^n),
        \qquad
        \int_{\mathbb T^n}\psi_A=0,
        \qquad
        c_A\in\mathbb R,
\]
and
\[
        \det(A+D^2\psi_A)
        =
        \exp\{-a\cdot D\psi_A+V-c_A\},
        \quad
        A+D^2\psi_A>0
        \quad\text{on }\mathbb T^n.
\]
\end{thm}

We prove the theorem by the continuity method.  The bound
\(D^2\psi\ge-A\) and periodicity give uniform \(C^0\) and \(C^1\)
estimates.  For the \(C^2\) estimate, the drift term in the full linearized
operator cancels the third-order term arising from the largest eigenvalue.
The higher-order estimates then follow from Evans--Krylov and Schauder
theory.

The normalizing constant satisfies
\[
        \min_{\mathbb T^n}V-\log\det A
        \le
        c_A
        \le
        \max_{\mathbb T^n}V-\log\det A.
\]
For fixed \(A\) and \(a\), it is \(1\)-Lipschitz with respect to \(V\) in
\(L^\infty\), and for fixed \(a\) and \(V\), it is decreasing with respect
to \(A\) in the Loewner order.

For comparison with weighted Monge--Amp\`ere equations on compact Hessian
manifolds, set
\[
        \Phi_A(x)=\frac12x^\top A x+\psi(x),
        \qquad
        p(x)=D\Phi_A(x)=Ax+D\psi(x).
\]
Then \eqref{eq:intro-cell} can be written as
\[
        e^{a\cdot p(x)}\det Dp(x)
        =
        e^{a\cdot Ax+V(x)-c},
\]
and
\[
        p(x+k)=p(x)+Ak,
        \qquad k\in\mathbb Z^n.
\]
This is related to the weighted Monge--Amp\`ere equations studied by Hultgren
and \"Onnheim \cite{HultgrenOnnheim19}.  Here, however, the source and target
weights transform by the same character under lattice translations rather
than being invariant, so their invariant-measure formulation does not apply
directly.

The cell problem also gives the compatibility condition for
quadratic-periodic solutions.  If
\[
        u(x)=\frac12x^\top A x+\ell\cdot x+\psi(x)+C,
\]
then the nonperiodic part of the exponent in \eqref{eq:intro-main} is
\((b-Aa)\cdot x\).  Hence such a solution can exist only when
\[
        b=Aa.
\]
The same condition is still necessary under a subquadratic asymptotic
assumption.  For prescribed densities converging to periodic data, related
quadratic-periodic asymptotics were obtained by Teixeira--Zhang
\cite{TeixeiraZhang16} and, more recently, by Qi--Bao \cite{QiBao25}.

\begin{thm}
\label{thm:bounded-classification}
Let \(V\in C^\infty(\mathbb T^n)\), and let
\(u\in C^\infty(\mathbb R^n)\) be a strictly convex solution of
\[
        \det D^2u
        =
        \exp\{-a\cdot Du+b\cdot x+V(x)-c_0\}
        \quad\text{in }\mathbb R^n.
\]
Suppose that, for some \(A=A^\top>0\) and \(\ell\in\mathbb R^n\),
\[
        f(x)=u(x)-\frac12x^\top A x-\ell\cdot x
\]
satisfies
\[
        \sup_{B_R}|f|=o(R^2)
        \quad\text{as }R\to\infty.
\]
Then \(b=Aa\).

If, in addition, \(f\in L^\infty(\mathbb R^n)\), then
\[
        c_0+a\cdot\ell=c_A
\]
and
\[
        u(x)
        =
        \frac12x^\top A x+\ell\cdot x+\psi_A(x)+C
\]
for some \(C\in\mathbb R\).
\end{thm}

For the first conclusion we compare the exponential growth caused by
\(b-Aa\neq0\) with the polynomial growth of \(Du(B_R)\).  In the bounded
case, compactness of integer translates and the strong comparison principle
give periodicity; the cell-problem uniqueness then finishes the proof.

The paper is organized as follows.  Section~\ref{sec:constant-coupled-drift}
proves the whole-space alternatives.  Section~\ref{sec:periodic-cell} treats
the periodic cell problem, and
Section~\ref{sec:periodic-correctors-rigidity} proves the compatibility and
bounded-corrector rigidity results.

\vspace{1em}
\noindent\textbf{Acknowledgments.}
The author is deeply grateful to his postdoctoral mentor, Prof.~Antonio De Rosa, for generous support and encouragement.

This work was funded by the European Union through the European Research
Council (ERC), under the Starting Grant ``ANGEVA'' (grant agreement
No.~101076411).  Views and opinions expressed are, however, those of the
author only and do not necessarily reflect those of the European Union or the
European Research Council.  Neither the European Union nor the granting
authority can be held responsible for them.

\section{The constant-medium equation}
\label{sec:constant-coupled-drift}

We prove Theorem~\ref{thm:linear-dichotomy} for the constant-medium equation
\begin{equation}\label{eq:LD-again}
        \det D^2u
        =
        \exp\{-a\cdot Du+b\cdot x-c_0\},
        \quad
        D^2u>0,
        \quad
        x\in\mathbb R^n.
\end{equation}
Throughout this section, strict convexity means \(D^2u>0\).

If \(a=b=0\), then
\eqref{eq:LD-again} reduces to the constant-density Monge--Amp\`ere equation
\[
        \det D^2u=e^{-c_0}
        \quad\text{in }\mathbb R^n.
\]
The J\"orgens--Calabi--Pogorelov theorem therefore implies that every entire
smooth strictly convex solution is a quadratic polynomial; see
\cite{Jorgens,Calabi,Pogorelov72}.  Its constant Hessian matrix necessarily satisfies
\(\det D^2u=e^{-c_0}\).

Before treating the negative-coupling case, we record an elementary consequence
of strict convexity.
\begin{lem}
\label{lem:properness-entire-convex}
Let \(u\in C^\infty(\mathbb R^n)\) be strictly convex and assume that
\(Du(0)=0\).  Then
\[
        u(x)\rightarrow+\infty
        \quad\text{as }|x|\rightarrow\infty.
\]
In particular, every sublevel set of \(u\) is bounded.
\end{lem}

\begin{proof}
For \(e\in \mathbb S^{n-1}\), set
\[
        g_e(t)=u(te).
\]
Then
\[
        g_e''(t)=e^\top D^2u(te)e>0,
        \quad
        g_e'(0)=0.
\]
Hence \(g_e'(1)>0\) for every \(e\in \mathbb S^{n-1}\).  Since
\[
        e\longmapsto g_e'(1)=Du(e)\cdot e
\]
is continuous on the compact unit sphere,
\[
        \sigma:=\min_{|e|=1}g_e'(1)>0.
\]
For \(r\ge1\), convexity gives
\[
        u(re)=g_e(r)
        \ge
        g_e(1)+(r-1)g_e'(1)
        \ge
        \min_{|e|=1}u(e)+\sigma(r-1).
\]
The right-hand side tends to \(+\infty\) uniformly in \(e\).
\end{proof}

\begin{prop}
\label{prop:negative-coupling}
Assume that \(a\neq0\) and \(a\cdot b<0\).  Then
\eqref{eq:LD-again} admits no entire smooth strictly convex solution.
\end{prop}

\begin{proof}
This is the nonexistence theorem of Wu--Xu
\cite[Theorem~1.3]{WuXu21}.  We include the argument since the same auxiliary
function will be used in the borderline case.

Suppose that a solution exists.  After translating the base point, subtracting
a supporting affine function, and changing the constant \(c_0\), we may assume
\[
        u(0)=3,
        \quad
        Du(0)=0,
        \quad
        u>0.
\]
For \(C>3\), set
\[
        \Omega_C=\{u<C\},
        \quad
        \eta=C-u,
        \quad
        \kappa=3C.
\]
By Lemma~\ref{lem:properness-entire-convex}, the domain \(\Omega_C\) is
bounded.

Define
\[
        Q
        =
        \exp\left(-\frac{\kappa}{\eta}\right)
        e^{-b\cdot x}\det D^2u
        =
        \exp\left(
        -\frac{\kappa}{\eta}
        -a\cdot Du
        -c_0
        \right)
        \quad\text{in }\Omega_C.
\]
The function \(Q\) extends continuously by zero to
\(\partial\Omega_C\).  It therefore attains its maximum at an interior point
\(x_C\in\Omega_C\).

At \(x_C\),
\begin{equation}\label{eq:sublevel-first-order}
        0=(\log Q)_i
        =
        -\frac{\kappa u_i}{\eta^2}
        -a_\ell u_{\ell i}.
\end{equation}
Contracting \eqref{eq:sublevel-first-order} with \(a_i\) and with
\(u^{ij}u_j\), respectively, gives
\begin{equation}\label{eq:sublevel-first-identities}
        a^\top D^2u\,a
        =
        -\frac{\kappa}{\eta^2}a\cdot Du,
        \quad
        -a\cdot Du
        =
        \frac{\kappa}{\eta^2}u^{ij}u_i u_j
        \ge0.
\end{equation}
At the same point,
\[
\begin{aligned}
        0
        &\ge
        u^{ij}(\log Q)_{ij} \\
        &=
        a^\top D^2u\,a
        -a\cdot b
        -\frac{\kappa n}{\eta^2}
        -\frac{2\kappa}{\eta^3}u^{ij}u_i u_j.
\end{aligned}
\]
Using \eqref{eq:sublevel-first-identities}, we obtain
\begin{equation}\label{eq:sublevel-key}
        0
        \ge
        \left(
        \frac{\kappa}{\eta^2}
        -\frac2\eta
        \right)(-a\cdot Du)
        -a\cdot b
        -\frac{\kappa n}{\eta^2}.
\end{equation}

Since \(0<u<C\) and \(\kappa=3C\),
\[
        \frac{\kappa}{\eta^2}-\frac2\eta
        =
        \frac{\kappa-2\eta}{\eta^2}
        =
        \frac{C+2u}{(C-u)^2}
        >0.
\]
The first term in \eqref{eq:sublevel-key} is therefore nonnegative.  Since
\(a\cdot b<0\), it follows that
\[
        \eta(x_C)^2
        \le
        \frac{\kappa n}{-a\cdot b}
        =
        \frac{3Cn}{-a\cdot b}.
\]
Consequently,
\[
        \frac{\kappa}{\eta(x_C)}
        \ge
        c_1\sqrt C
\]
for some \(c_1>0\) depending only on \(n\) and \(-a\cdot b\).

Set
\[
        X=-a\cdot Du(x_C)\ge0.
\]
Equation \eqref{eq:sublevel-key} also gives
\[
        \left(
        \frac{\kappa}{\eta(x_C)^2}
        -\frac2{\eta(x_C)}
        \right)X
        \le
        a\cdot b+\frac{\kappa n}{\eta(x_C)^2}
        <
        \frac{\kappa n}{\eta(x_C)^2}.
\]
Thus
\[
        X
        <
        \frac{\kappa n}{\kappa-2\eta(x_C)}
        <3n.
\]
It follows that
\[
\begin{aligned}
        \log Q(x_C)
        &=
        -\frac{\kappa}{\eta(x_C)}
        -a\cdot Du(x_C)
        -c_0 \\
        &\le
        -c_1\sqrt C+3n-c_0,
\end{aligned}
\]
and therefore
\[
        Q(x_C)\rightarrow0
        \quad\text{as }C\rightarrow\infty.
\]
On the other hand,
\[
        Q(0)
        =
        \exp\left(-\frac{3C}{C-3}-c_0\right)
        \rightarrow e^{-3-c_0}>0.
\]
For sufficiently large \(C\), this contradicts
\(Q(0)\le Q(x_C)\).
\end{proof}

Assume now that \(a\cdot b=0\).  We first obtain a global oscillation
estimate for \(a\cdot Du\).

\begin{lem}
\label{lem:borderline}
Let \(u\) solve \eqref{eq:LD-again} and suppose that \(a\cdot b=0\).  Then
\[
        \left|
        a\cdot Du(x)-a\cdot Du(y)
        \right|
        \le 3n+3
        \quad
        \text{for all }x,y\in\mathbb R^n.
\]
\end{lem}

\begin{proof}
Fix a base point and perform an affine normalization so that
\[
        u(0)=3,
        \quad
        Du(0)=0,
        \quad
        u>0.
\]
For \(C>3\), use the same quantities
\[
        \Omega_C=\{u<C\},
        \quad
        \eta=C-u,
        \quad
        \kappa=3C,
\]
and the same auxiliary function \(Q\) as in the proof of
Proposition~\ref{prop:negative-coupling}.  Let \(x_C\) be its interior
maximum point.

Equations \eqref{eq:sublevel-first-identities} and
\eqref{eq:sublevel-key} remain valid.  Since \(a\cdot b=0\), we obtain
\[
        0
        \ge
        \left(
        \frac{\kappa}{\eta^2}
        -\frac2\eta
        \right)(-a\cdot Du)
        -\frac{\kappa n}{\eta^2}
\]
at \(x_C\).  Therefore
\[
        -a\cdot Du(x_C)
        \le
        \frac{\kappa n}{\kappa-2\eta(x_C)}
        <3n.
\]

For any fixed \(x\in\mathbb R^n\), once \(C>u(x)\), the inequality
\(Q(x)\le Q(x_C)\) gives
\[
        -a\cdot Du(x)
        \le
        -a\cdot Du(x_C)
        +\frac{\kappa}{C-u(x)}
        -\frac{\kappa}{\eta(x_C)}.
\]
Discarding the final nonpositive term and letting \(C\to\infty\), we find
\[
        -a\cdot Du(x)\le3n+3.
\]

To compare two arbitrary points, fix \(x_0\in\mathbb R^n\) and consider
\[
        v_{x_0}(y)
        =
        u(x_0+y)-u(x_0)-Du(x_0)\cdot y+3.
\]
By convexity,
\[
        v_{x_0}(0)=3,
        \quad
        Dv_{x_0}(0)=0,
        \quad
        v_{x_0}\ge3.
\]
The function \(v_{x_0}\) satisfies an equation of the form
\eqref{eq:LD-again} with the same vectors \(a,b\), the same coupling
\(a\cdot b=0\), and a possibly different constant.  Applying the preceding
one-sided estimate to \(v_{x_0}\) yields
\[
        a\cdot Du(x_0+y)-a\cdot Du(x_0)
        \ge-(3n+3).
\]
Interchanging the two points gives the reverse inequality.
\end{proof}

We use this oscillation bound in the following proposition.

\begin{prop}
\label{prop:bounded-null-rigidity}
Let \(n\ge2\), \(a\neq0\), and \(a\cdot b=0\).  Suppose that
\(u\in C^4(\mathbb R^n)\) is strictly convex, solves
\eqref{eq:LD-again}, and satisfies
\[
        \osc_{\mathbb R^n}(a\cdot Du)<\infty.
\]
Then no such solution exists.
\end{prop}

\begin{proof}
Suppose, to the contrary, that such a solution exists.  Fix a point
\(x_0\in\mathbb R^n\) and define
\[
        \widetilde u(x)
        =
        u(x+x_0)-u(x_0)-Du(x_0)\cdot x.
\]
Then
\[
        \widetilde u(0)=0,
        \quad
        D\widetilde u(0)=0,
        \quad
        \widetilde u\ge0
\]
by convexity.  The function \(\widetilde u\) also satisfies an equation of the
same form as \eqref{eq:LD-again}, with the same vectors \(a\) and \(b\), and
with
\(c_0\) replaced by
\[
        \widetilde c_0
        =
        c_0+a\cdot Du(x_0)-b\cdot x_0.
\]
The condition \(a\cdot b=0\) is unchanged, and
\[
        a\cdot D\widetilde u(x)
        =
        a\cdot Du(x+x_0)-a\cdot Du(x_0),
\]
so the oscillation assumption is unchanged as well.  Renaming
\(\widetilde u\) as \(u\), we may assume
\[
        u(0)=0,
        \quad
        Du(0)=0,
        \quad
        u\ge0.
\]
By Lemma~\ref{lem:properness-entire-convex},
\[
        u(x)\rightarrow+\infty
        \quad\text{as }|x|\rightarrow\infty.
\]

Set
\[
        w=a\cdot Du.
\]
Since \(w(0)=0\), the oscillation assumption implies
\[
        |w|\le W
        \quad\text{in }\mathbb R^n
\]
for some finite constant \(W\).

Write
\[
        g_{ij}=u_{ij},
        \quad
        (g^{ij})=(g_{ij})^{-1},
\]
and define the drifted linearized operator
\[
        \mathcal Lf
        =
        g^{ij}f_{ij}+a_kf_k.
\]
We use the summation convention and set
\[
        P=a_ra_su_{rs}=a^\top D^2u\,a>0,
        \quad
        T_{ij}=a_ku_{ijk}.
\]
Here and below,
\[
        |T|_g^2
        =
        g^{ip}g^{jq}T_{ij}T_{pq},
        \quad
        |\nabla P|_g^2
        =
        g^{ij}P_iP_j.
\]

Since
\[
        w_i=a_ku_{ki}=g_{ik}a_k,
\]
we have
\[
        P=g^{ij}w_iw_j=|\nabla w|_g^2,
        \qquad
        w_{ij}=a_ku_{kij}=T_{ij}.
\]

Taking the logarithm of \eqref{eq:LD-again} gives
\[
        \log\det(g_{ij})
        =
        -a\cdot Du+b\cdot x-c_0.
\]
Differentiating in the \(x_k\)-direction yields
\[
        g^{ij}u_{ijk}
        =
        b_k-a_\ell u_{\ell k}.
\]
Contracting with \(a_k\) and using \(a\cdot b=0\), we obtain
\begin{equation}\label{eq:null-trace-T}
        g^{ij}T_{ij}=-P.
\end{equation}
Since
\[
        a\cdot Dw=P,
\]
identity \eqref{eq:null-trace-T} also gives
\[
        \mathcal Lw
        =
        g^{ij}w_{ij}+a\cdot Dw
        =
        -P+P
        =
        0.
\]

Differentiating the logarithmic equation once more gives
\[
        g^{ij}u_{ijrs}
        -
        g^{ip}g^{jq}u_{ijr}u_{pqs}
        =
        -a_ku_{krs}.
\]
Multiplying by \(a_ra_s\), we find
\[
        g^{ij}P_{ij}
        -
        g^{ip}g^{jq}T_{ij}T_{pq}
        =
        -a_kP_k.
\]
Hence the first-order term in \(\mathcal L\) cancels the right-hand side,
and
\begin{equation}\label{eq:null-LP}
        \mathcal LP
        =
        g^{ip}g^{jq}T_{ij}T_{pq}
        =
        |T|_g^2.
\end{equation}
Differentiating the definition of \(P\) also gives
\begin{equation}\label{eq:null-gradient-P}
        P_i
        =
        a_ra_su_{rsi}
        =
        T_{ij}a_j.
\end{equation}

At a fixed point, choose a \(g\)-orthonormal basis whose first vector is
parallel to \(a\).  In this basis,
\[
        g_{ij}=\delta_{ij},
        \quad
        a=\sqrt P\,e_1.
\]
This choice is pointwise; no derivatives of the frame are taken.  Equations
\eqref{eq:null-LP}, \eqref{eq:null-trace-T}, and
\eqref{eq:null-gradient-P} become
\[
        \mathcal LP=|T|^2,
        \quad
        \tr T=-P,
\]
and
\begin{equation}\label{eq:null-P-gradient-frame}
        |\nabla P|_g^2
        =
        P\sum_iT_{1i}^2,
        \quad
        a\cdot DP
        =
        PT_{11}.
\end{equation}

Choose
\[
        \theta=\frac1{4(n-1)}.
\]
In particular,
\[
        0<\theta<1.
\]
The chain rule, \eqref{eq:null-LP}, and
\eqref{eq:null-P-gradient-frame} give
\[
\begin{aligned}
        \mathcal L(-P^{-\theta})
        &=
        \theta P^{-\theta-1}\mathcal LP
        -
        \theta(\theta+1)P^{-\theta-2}
        |\nabla P|_g^2 \\
        &=
        \theta P^{-\theta-1}
        \left(
        |T|^2-(\theta+1)\sum_iT_{1i}^2
        \right).
\end{aligned}
\]

Write
\[
        T=
        \begin{pmatrix}
        \tau&v^\top\\
        v&B
        \end{pmatrix},
        \quad
        \tau=T_{11},
\]
where
\[
        v=(T_{12},\ldots,T_{1n})
\]
and
\[
        B=(T_{\alpha\beta})_{2\le\alpha,\beta\le n}.
\]
Then
\[
        |T|^2=\tau^2+2|v|^2+|B|^2
\]
and
\[
        \sum_iT_{1i}^2=\tau^2+|v|^2.
\]
Consequently,
\[
\begin{aligned}
        |T|^2-(\theta+1)\sum_iT_{1i}^2
        &=
        -\theta\tau^2
        +(1-\theta)|v|^2
        +|B|^2 \\
        &\ge
        -\theta\tau^2+|B|^2.
\end{aligned}
\]
By \eqref{eq:null-trace-T},
\[
        \tr B=-P-\tau.
\]
Since \(B\) is an \((n-1)\times(n-1)\) symmetric matrix, the
Cauchy--Schwarz inequality gives
\[
        |B|^2
        \ge
        \frac{(\tr B)^2}{n-1}
        =
        \frac{(P+\tau)^2}{n-1}.
\]
Thus
\begin{equation}\label{eq:null-coercive-P}
        \mathcal L(-P^{-\theta})
        \ge
        \theta P^{-\theta-1}
        \left(
        \frac{(P+\tau)^2}{n-1}
        -\theta\tau^2
        \right).
\end{equation}

For \(\varepsilon>0\), define
\[
        F_\varepsilon
        =
        -P^{-\theta}-\varepsilon u.
\]
Since \(P>0\), one has \(-P^{-\theta}\le0\).  Together with
\(u(x)\to+\infty\), this gives
\[
        F_\varepsilon(x)\rightarrow-\infty
        \quad\text{as }|x|\rightarrow\infty.
\]
Thus \(F_\varepsilon\) attains a global maximum at some point
\(x_\varepsilon\).

At \(x_\varepsilon\),
\[
        DF_\varepsilon=0,
\]
and hence
\[
        \theta P^{-\theta-1}DP
        =
        \varepsilon Du.
\]
Choose at \(x_\varepsilon\) the adapted \(g\)-orthonormal basis used above.
Contracting the preceding identity with \(a\) and using
\eqref{eq:null-P-gradient-frame}, we obtain
\[
        \theta P^{-\theta-1}(P\tau)
        =
        \varepsilon w.
\]
Thus
\begin{equation}\label{eq:null-tau-bound}
        \tau
        =
        \frac{\varepsilon}{\theta}wP^\theta,
        \quad
        |\tau|
        \le
        \frac{\varepsilon W}{\theta}P^\theta.
\end{equation}
All quantities in the rest of the argument are evaluated at
\(x_\varepsilon\).

Suppose first that
\[
        |\tau|\ge\frac P2.
\]
Then \eqref{eq:null-tau-bound} gives
\[
        \frac P2
        \le
        \frac{\varepsilon W}{\theta}P^\theta,
\]
and hence
\begin{equation}\label{eq:null-case-one}
        P^{1-\theta}
        \le
        \frac{2W}{\theta}\varepsilon.
\end{equation}

Suppose instead that
\[
        |\tau|<\frac P2.
\]
Then
\[
        P+\tau\ge\frac P2
\]
and
\[
        \tau^2<\frac{P^2}{4}.
\]
Therefore
\[
\begin{aligned}
        \frac{(P+\tau)^2}{n-1}-\theta\tau^2
        &\ge
        \frac{P^2}{4(n-1)}
        -
        \frac{\theta P^2}{4} \\
        &=
        \frac{P^2}{4}
        \left(
        \frac1{n-1}-\theta
        \right).
\end{aligned}
\]
Since
\[
        \theta=\frac1{4(n-1)},
\]
the right-hand side equals
\[
        \frac{3P^2}{16(n-1)}
        \ge
        \frac{P^2}{8(n-1)}.
\]
Estimate \eqref{eq:null-coercive-P} gives
\[
        \mathcal L(-P^{-\theta})
        \ge
        \frac{\theta}{8(n-1)}P^{1-\theta}.
\]

At the maximum point,
\[
        DF_\varepsilon(x_\varepsilon)=0
\]
and
\[
        D^2F_\varepsilon(x_\varepsilon)\le0.
\]
The first-order part of \(\mathcal L\) therefore vanishes there, and
\[
        \mathcal LF_\varepsilon(x_\varepsilon)
        =
        g^{ij}(F_\varepsilon)_{ij}(x_\varepsilon)
        \le0.
\]
At the same point,
\[
        \mathcal Lu
        =
        g^{ij}u_{ij}+a\cdot Du
        =
        n+w
        \le
        n+W.
\]
Hence
\[
\begin{aligned}
        0
        &\ge
        \mathcal LF_\varepsilon=
        \mathcal L(-P^{-\theta})
        -
        \varepsilon\mathcal Lu \\
        &\ge
        \frac{\theta}{8(n-1)}P^{1-\theta}
        -
        \varepsilon(n+W).
\end{aligned}
\]
Hence
\begin{equation}\label{eq:null-case-two}
        P^{1-\theta}
        \le
        \frac{8(n-1)(n+W)}{\theta}\varepsilon.
\end{equation}

Combining \eqref{eq:null-case-one} and
\eqref{eq:null-case-two}, we obtain
\[
        P(x_\varepsilon)^{1-\theta}
        \le
        C\varepsilon,
\]
where \(C\) is independent of \(\varepsilon\).  Since
\(1-\theta>0\), this implies
\begin{equation}\label{eq:null-P-small}
        P(x_\varepsilon)\rightarrow0
        \quad\text{as }\varepsilon\downarrow0.
\end{equation}

On the other hand, maximality of \(F_\varepsilon\), together with
\(u(0)=0\) and \(u\ge0\), gives
\[
\begin{aligned}
        -P(x_\varepsilon)^{-\theta}
        \ge
        -P(x_\varepsilon)^{-\theta}
        -
        \varepsilon u(x_\varepsilon) =
        F_\varepsilon(x_\varepsilon)\ge
        F_\varepsilon(0) =
        -P(0)^{-\theta}.
\end{aligned}
\]
Thus
\[
        P(x_\varepsilon)\ge P(0)>0.
\]
This contradicts \eqref{eq:null-P-small}.
\end{proof}

\begin{rem}
\label{rem:LiXu-comparison}
When \(b=0\), Proposition~\ref{prop:bounded-null-rigidity} gives another proof
of the Li--Xu rigidity theorem \cite{LiXu09} in the smooth strictly convex
case.  The quantity \(P\) used above is directly related to their affine
invariant.  Indeed, let \(f=u^*\).  Then
\[
        \det D^2f=e^{a\cdot y+c_0}.
\]
If
\[
        \rho=(\det D^2f)^{-1/(n+2)},
        \qquad
        \Phi=\frac{f^{ij}\rho_i\rho_j}{\rho^2},
\]
as in \cite{LiXu09}, then at the dual point \(y=Du(x)\),
\[
        \Phi(y)=\frac{1}{(n+2)^2}a^\top D^2u(x)a
        =\frac{P(x)}{(n+2)^2}.
\]
Thus, up to a constant factor, \(P\) is the same scalar quantity in the
original variables.  In the null-coupling case it satisfies
\[
        \mathcal Lw=0,
        \qquad
        P=|\nabla w|_{D^2u}^2,
        \qquad
        \mathcal LP=|T|_{D^2u}^2,
        \qquad
        \operatorname{tr}_{D^2u}T=-P,
\]
which leads to the maximum-principle argument above.
\end{rem}

\begin{rem}
The mixed null case \(a\neq0\), \(b\neq0\), \(a\cdot b=0\) is not reduced to
the pure gradient-drift case by the natural changes of variables preserving
entire convex gradient graphs.  If \(S\in GL(n,\R)\) and
\[
        v(z)=u(Sz)+q\cdot z+d,
\]
then the drift vectors become
\[
        a'=S^{-1}a,
        \qquad
        b'=S^\top b,
        \qquad
        a'\cdot b'=a\cdot b.
\]
Hence an invertible change of the base variables cannot eliminate a nonzero
spatial drift.  The Legendre transform exchanges the two drift vectors, but
it is defined on \(Du(\R^n)\), which need not be all of \(\R^n\).  Therefore
it does not in general preserve the class of entire solutions considered
here.
\end{rem}

Lemma~\ref{lem:borderline} and Proposition~\ref{prop:bounded-null-rigidity}
give the null-coupling obstruction.

\begin{prop}
\label{prop:null-nonexistence}
Assume that
\[
        a\neq0,
        \quad
        a\cdot b=0.
\]
Then \eqref{eq:LD-again} admits no entire smooth strictly convex solution.
\end{prop}

\begin{proof}
Assume first that \(n\ge2\).  Lemma~\ref{lem:borderline} gives
\[
        \osc_{\mathbb R^n}(a\cdot Du)<\infty.
\]
The conclusion therefore follows from
Proposition~\ref{prop:bounded-null-rigidity}.

If \(n=1\), then \(a\neq0\) and \(ab=0\) imply \(b=0\).  The equation is
\[
        u''=e^{-au'-c_0}.
\]
Consequently,
\[
        \bigl(e^{au'}\bigr)'
        =
        ae^{-c_0}\neq0.
\]
Thus \(e^{au'}\) is a nonconstant affine function of \(x\), which cannot
remain positive on all of \(\mathbb R\).  This is a contradiction.
\end{proof}

It remains to construct non-quadratic entire solutions in the two solvable regimes.

Suppose first that \(a=0\) and \(b\neq0\).  After an orthogonal change of
coordinates, write
\[
        b=\beta e_n,
        \quad
        \beta=|b|>0.
\]
Then
\[
        u(x',x_n)
        =
        \frac12|x'|^2
        +
        \frac{e^{-c_0}}{\beta^2}e^{\beta x_n}
\]
is smooth and strictly convex, and
\[
        \det D^2u
        =
        e^{\beta x_n-c_0}
        =
        e^{b\cdot x-c_0}.
\]
It is plainly non-quadratic.

Assume now that
\[
        a\cdot b>0.
\]
Choose a symmetric positive definite matrix \(A\) satisfying
\[
        Aa=b.
\]
For \(n=1\), simply take \(A=b/a>0\).  For \(n\ge2\), use an orthonormal basis
whose first vector is \(e_1=a/|a|\), and write
\[
        \frac{b}{|a|}
        =
        \begin{pmatrix}
        \alpha\\
        \beta
        \end{pmatrix},
        \quad
        \alpha=\frac{a\cdot b}{|a|^2}>0,
        \quad
        \beta\in\mathbb R^{n-1}.
\]
For any
\[
        \lambda>\frac{|\beta|^2}{\alpha},
\]
the matrix
\[
        A=
        \begin{pmatrix}
        \alpha&\beta^\top\\
        \beta&\lambda I_{n-1}
        \end{pmatrix}
\]
is positive definite by the Schur complement criterion and satisfies
\(Aa=b\).

First note that the condition \(Aa=b\) also gives quadratic solutions.  Since
\(a\neq0\), choose \(\ell_0\in\mathbb R^n\) so that
\[
        a\cdot\ell_0=-c_0-\log\det A.
\]
Then
\[
        u_0(x)=\frac12 x^\top A x+\ell_0\cdot x
\]
solves \eqref{eq:LD-again}.  To show that the positive-coupling regime also
contains non-quadratic solutions, choose \(\ell\in\mathbb R^n\) such that
\[
        \gamma=a\cdot\ell\neq0,
\]
and set
\[
        \mu=\ell^\top A^{-1}\ell>0,
        \quad
        K_0=\frac{e^{-c_0}}{\det A}.
\]
Fix \(C_0>0\), and define a smooth function \(H\) on \(\mathbb R\), up to an
additive constant, by
\[
        e^{\gamma H'(t)}
        =
        K_0+C_0e^{-\gamma t/\mu}.
\]
Differentiating gives
\[
        H''(t)
        =
        -\frac1\mu
        \frac{C_0e^{-\gamma t/\mu}}
        {K_0+C_0e^{-\gamma t/\mu}},
\]
and hence
\begin{equation}\label{eq:front-convexity}
        1+\mu H''(t)
        =
        \frac{K_0}
        {K_0+C_0e^{-\gamma t/\mu}}
        =
        K_0e^{-\gamma H'(t)}
        >0.
\end{equation}

Define
\[
        u(x)
        =
        \frac12 x^\top A x+H(\ell\cdot x).
\]
Then
\[
        D^2u
        =
        A+H''(\ell\cdot x)\,\ell\otimes\ell.
\]
Since \(A>0\), the rank-one update criterion and
\eqref{eq:front-convexity} imply that \(D^2u>0\).  The matrix determinant
lemma gives
\[
        \det D^2u
        =
        \det A\,
        \bigl(1+\mu H''(\ell\cdot x)\bigr).
\]
Since \(A\) is symmetric and \(Aa=b\),
\[
\begin{aligned}
        -a\cdot Du+b\cdot x
        &=
        -a\cdot
        \bigl(Ax+H'(\ell\cdot x)\ell\bigr)
        +b\cdot x \\
        &=
        -\gamma H'(\ell\cdot x).
\end{aligned}
\]
Using \eqref{eq:front-convexity} and the definition of \(K_0\), we conclude
that
\[
\begin{aligned}
        \det D^2u
        &=
        \det A\,K_0e^{-\gamma H'(\ell\cdot x)} \\
        &=
        \exp\{-a\cdot Du+b\cdot x-c_0\}.
\end{aligned}
\]
Because \(C_0>0\), the function \(H\) is not quadratic, and neither is \(u\).

\begin{rem}
Notice that the positive-coupling construction cannot remain uniformly convex as
\(a\cdot b\downarrow0\) with \(a\neq0\).  Indeed, every positive definite
matrix \(A\) satisfying \(Aa=b\) obeys
\[
        a^\top A a=a\cdot b,
        \quad
        \lambda_{\min}(A)
        \le
        \frac{a\cdot b}{|a|^2}.
\]
Thus the quadratic background degenerates as the null threshold is
approached.
\end{rem}

\begin{proof}[Proof of Theorem~\ref{thm:linear-dichotomy}]
If \(a=b=0\), the classical J\"orgens--Calabi--Pogorelov theorem gives part
\textup{(i)}.

Suppose that \(a\neq0\) and \(a\cdot b\le0\).  If \(a\cdot b<0\),
Proposition~\ref{prop:negative-coupling} gives nonexistence.  If
\(a\cdot b=0\), Proposition~\ref{prop:null-nonexistence} gives
nonexistence.  This proves part \textup{(ii)}.

The two explicit constructions above prove part \textup{(iii)} and complete
the proof.
\end{proof}
\section{The periodic cell problem}
\label{sec:periodic-cell}

Throughout this section,
\(\T^n=\R^n/\mathbb Z^n\)
is normalized to have volume one.  Let
\[
        A=A^\top>0,
        \quad
        a\in\R^n,
        \quad
        V\in C^\infty(\T^n).
\]
We prove Theorem~\ref{thm:cell}, namely the existence and uniqueness of a
constant \(c_A\in\R\) and a zero-average function
\(\psi_A\in C^\infty(\T^n)\) satisfying
\begin{equation}\label{eq:cell-again}
        \det(A+D^2\psi_A)
        =
        \exp\{-a\cdot D\psi_A+V-c_A\},
        \quad
        A+D^2\psi_A>0
        \quad\text{on }\T^n.
\end{equation}
We write \(c_A=c(A;a,V)\) when the dependence on the other data is relevant.

Consider the continuity path
\begin{equation}\label{eq:cell-path}
        \log\det(A+D^2\psi)
        +t\,a\cdot D\psi
        -tV+c=0,
        \quad
        t\in[0,1].
\end{equation}
\subsection{Comparison and uniqueness}

We use the following comparison principle.

\begin{lem}[Strong comparison principle]
\label{lem:strong-comparison}
Let \(\Omega\subset\R^n\) be connected, let \(N=N^\top\) be a smooth matrix
field on \(\Omega\), and let \(u_1,u_2\in C^2(\Omega)\) satisfy
\[
        N+D^2u_i>0,
        \quad i=1,2.
\]
Assume that
\[
        \log\det(N+D^2u_1)+a\cdot Du_1
        =
        \log\det(N+D^2u_2)+a\cdot Du_2
        \quad\text{in }\Omega.
\]
If
\[
        u_1\le u_2
        \quad\text{in }\Omega
\]
and equality holds at an interior point, then
\[
        u_1\equiv u_2
        \quad\text{in }\Omega.
\]
\end{lem}

\begin{proof}
Set
\[
        w=u_1-u_2.
\]
Subtracting the two equations and integrating the logarithmic determinant
along the segment \(u_2+s w\), \(s\in[0,1]\), gives
\[
        B^{ij}w_{ij}+a\cdot Dw=0,
\]
where
\[
        B^{ij}(x)
        =
        \int_0^1
        \bigl(
        N+D^2u_2+sD^2w
        \bigr)^{-1}_{ij}
        \,\mathrm{d}s.
\]
The positive definite cone is convex, so every matrix inside the integral is
positive definite.  Thus the resulting linear operator is uniformly elliptic
on compact subsets of \(\Omega\).  Since \(w\le0\) and \(w\) attains an
interior maximum, the strong maximum principle gives \(w\equiv0\).
\end{proof}

\begin{lem}
\label{lem:cell-uniqueness}
Suppose that \((\psi_1,c_1)\) and \((\psi_2,c_2)\) are two smooth solutions of
\eqref{eq:cell-again}.  Then \(c_1=c_2\), and \(\psi_1-\psi_2\) is
constant.  In particular, if both functions have zero average, then
\(\psi_1=\psi_2\).
\end{lem}

\begin{proof}
Let
\[
        w=\psi_1-\psi_2.
\]
At a maximum point \(x_+\) of \(w\),
\[
        Dw(x_+)=0,
        \quad
        D^2w(x_+)\le0.
\]
Hence
\[
        A+D^2\psi_1(x_+)
        \le
        A+D^2\psi_2(x_+).
\]
The determinant is monotone on the positive definite cone, and
\(D\psi_1(x_+)=D\psi_2(x_+)\).  Evaluating the equations at \(x_+\) therefore
gives
\(c_1\ge c_2\).
Applying the same argument at a minimum point of \(w\) gives
\(c_1\le c_2\).
Hence \(c_1=c_2\).

Let
\[
        m=\max_{\T^n}(\psi_1-\psi_2).
\]
The equation is invariant under adding constants to \(\psi\), so
\(\psi_2+m\) satisfies the same equation as \(\psi_2\), and
\[
        \psi_1\le\psi_2+m
\]
and equality holds at an interior point of the periodic lift to \(\R^n\).
Lemma~\ref{lem:strong-comparison} gives
\[
        \psi_1\equiv\psi_2+m.
\]
If both functions have zero average, then \(m=0\).
\end{proof}

We record some elementary properties of the normalizing constant.

\begin{lem}
\label{lem:cell-constant-comparison}
The following statements hold.

\begin{enumerate}[label=\textnormal{(\roman*)}]
\item
Suppose that \((\psi,c)\) solves
\[
        \det(A+D^2\psi)
        =
        e^{-t a\cdot D\psi+tV-c},
        \quad
        A+D^2\psi>0,
\]
for some \(t\in[0,1]\).  Then
\begin{equation}\label{eq:cell-c-sharp-bounds}
        t\min_{\T^n}V-\log\det A
        \le
        c
        \le
        t\max_{\T^n}V-\log\det A.
\end{equation}

\item
For fixed \(A\) and \(a\), suppose that \((\psi_i,c_i)\) solves the cell
problem with potential \(V_i\), \(i=1,2\).  Then
\begin{equation}\label{eq:cell-V-comparison}
        \min_{\T^n}(V_1-V_2)
        \le
        c_1-c_2
        \le
        \max_{\T^n}(V_1-V_2).
\end{equation}
In particular,
\begin{equation}\label{eq:cell-c-Linfty-stability}
        |c_1-c_2|
        \le
        \|V_1-V_2\|_{L^\infty(\T^n)}.
\end{equation}

\item
Fix \(a\) and \(V\).  Suppose that \(A_1,A_2>0\) are symmetric and
\[
        A_1\le A_2
\]
in the sense of quadratic forms.  If \(c_{A_i}\) denotes the corresponding
normalizing constant, then
\[
        c_{A_1}\ge c_{A_2}.
\]
If \(A_2-A_1>0\), then the inequality is strict.
\end{enumerate}
\end{lem}

\begin{proof}
Let \(x_+\) be a maximum point of \(\psi\).  Then
\[
        D\psi(x_+)=0,
        \quad
        D^2\psi(x_+)\le0.
\]
Thus
\[
        A+D^2\psi(x_+)\le A,
\]
and hence
\[
        e^{tV(x_+)-c}
        =
        \det(A+D^2\psi(x_+))
        \le
        \det A.
\]
Therefore
\[
        c
        \ge
        tV(x_+)-\log\det A
        \ge
        t\min_{\T^n}V-\log\det A.
\]
The upper bound follows similarly by evaluating at a minimum point of
\(\psi\), where \(D^2\psi\ge0\).  This proves
\eqref{eq:cell-c-sharp-bounds}.

For the second statement, let \(x_+\) be a maximum point of
\(\psi_1-\psi_2\).  At \(x_+\),
\[
        D\psi_1=D\psi_2,
        \quad
        D^2\psi_1\le D^2\psi_2.
\]
It follows that
\[
        V_1(x_+)-c_1
        \le
        V_2(x_+)-c_2,
\]
and hence
\[
        c_1-c_2
        \ge
        V_1(x_+)-V_2(x_+)
        \ge
        \min_{\T^n}(V_1-V_2).
\]
Applying the same argument at a minimum point of
\(\psi_1-\psi_2\) gives the opposite bound in
\eqref{eq:cell-V-comparison}.  Estimate
\eqref{eq:cell-c-Linfty-stability} follows immediately.

Finally, let \(\psi_i\) be the solution corresponding to \(A_i\), and let
\(x_+\) be a maximum point of \(\psi_1-\psi_2\).  Then
\[
        D^2\psi_1(x_+)\le D^2\psi_2(x_+),
\]
so
\[
        A_1+D^2\psi_1(x_+)
        \le
        A_2+D^2\psi_2(x_+).
\]
Comparing the two equations at \(x_+\) gives
\(c_{A_1}\ge c_{A_2}\).
If \(A_2-A_1>0\), the matrix inequality at \(x_+\) is strict, and the strict
monotonicity of the determinant on the positive definite cone gives
\(c_{A_1}>c_{A_2}\).
\end{proof}

For fixed \(A\) and \(a\), we write \(c_A[V]\) for the normalizing constant
associated with the potential \(V\).
Lemma~\ref{lem:cell-constant-comparison} also gives
\[
        V_1\le V_2
        \quad\Longrightarrow\quad
        c_A[V_1]\le c_A[V_2],
\]
and
\[
        c_A[V+\lambda]=c_A[V]+\lambda
        \quad
        \text{for every constant }\lambda\in\R.
\]

\subsection{A priori estimates}

For \(t\in[0,1]\), suppose that \((\psi,c)\) is a smooth solution of
\eqref{eq:cell-path}, normalized by
\[
        \int_{\T^n}\psi=0.
\]
Set
\[
        M=A+D^2\psi.
\]
All estimates below are independent of \(t\).  The first one follows from
periodicity and semiconvexity.

\begin{lem}
\label{lem:periodic-semiconvexity}
There exists a constant \(C=C(A,n)\) such that every solution along the
continuity path satisfies
\begin{equation}\label{eq:C1-cell}
        \|D\psi\|_{L^\infty(\T^n)}
        \le C
\end{equation}
and
\begin{equation}\label{eq:C0-cell}
        \|\psi\|_{L^\infty(\T^n)}
        \le C.
\end{equation}
\end{lem}

\begin{proof}
Since
\[
        A+D^2\psi>0,
\]
we have
\[
        D^2\psi\ge-A
\]
as quadratic forms.  In particular,
\[
        \psi_{ii}\ge-A_{ii}.
\]

Fix all variables except \(x_i\), and let
\[
        f(s)
        =
        \psi(x_1,\ldots,x_{i-1},s,x_{i+1},\ldots,x_n).
\]
Then \(f\) is \(1\)-periodic and satisfies
\[
        f''\ge-A_{ii}.
\]
We claim that
\[
        \|f'\|_{L^\infty([0,1])}\le A_{ii}.
\]
Indeed, for \(0\le s<t\le1\),
\[
        f'(t)-f'(s)\ge-A_{ii}(t-s).
\]
Applying the same inequality to the periodically extended function on the
interval \([t,s+1]\) gives
\[
        f'(s)-f'(t)\ge-A_{ii}(s+1-t).
\]
Hence
\[
        \operatorname{osc}_{[0,1]}f'\le A_{ii}.
\]
Since
\[
        \int_0^1 f'(s)\,\mathrm{d}s=0,
\]
the asserted bound follows.

Applying this argument in each coordinate direction gives the explicit bound
\[
        \|D\psi\|_{L^\infty(\T^n)}
        \le
        \left(\sum_{i=1}^n A_{ii}^2\right)^{1/2}
        \le
        \sqrt n\,\lambda_{\max}(A),
\]
and hence \eqref{eq:C1-cell}.  Since \(\psi\) has zero average, its minimum is
nonpositive and its maximum is nonnegative.  The gradient bound therefore
implies
\[
        \operatorname{osc}_{\T^n}\psi
        \le
        \operatorname{diam}(\T^n)
        \|D\psi\|_{L^\infty(\T^n)},
\]
which gives \eqref{eq:C0-cell}.
\end{proof}

We also use the following determinant identity.

\begin{lem}
\label{lem:periodic-determinant-identity}
For every smooth periodic \(\psi\),
\[
        \int_{\T^n}
        \det(A+D^2\psi)\,\mathrm{d}x
        =
        \det A.
\]
\end{lem}

\begin{proof}
For \(s\in[0,1]\), define
\[
        I(s)
        =
        \int_{\T^n}
        \det(A+sD^2\psi)\,\mathrm{d}x.
\]
Since
\[
        A+sD^2\psi
        =
        D^2\left(
        \frac12 x^\top A x+s\psi
        \right),
\]
its cofactor matrix is divergence-free:
\[
        \partial_i
        \cof(A+sD^2\psi)^{ij}
        =0.
\]
Therefore
\[
\begin{aligned}
        I'(s)
        &=
        \int_{\T^n}
        \cof(A+sD^2\psi)^{ij}\psi_{ij}
        \,\mathrm{d}x \\
        &=
        \int_{\T^n}
        \partial_i
        \left(
        \cof(A+sD^2\psi)^{ij}\psi_j
        \right)
        \,\mathrm{d}x \\
        &=0.
\end{aligned}
\]
Thus \(I(1)=I(0)=\det A\).
\end{proof}

Combining \eqref{eq:cell-path} with
Lemma~\ref{lem:periodic-determinant-identity}, we obtain the exact identity
\[
        c
        =
        \log\left(
        \frac{
        \displaystyle
        \int_{\T^n}
        e^{-t a\cdot D\psi+tV}\,\mathrm{d}x
        }{\det A}
        \right).
\]
For the \textit{a priori} estimates we use instead the pointwise bounds
\eqref{eq:cell-c-sharp-bounds}.  Together with
Lemma~\ref{lem:periodic-semiconvexity}, they give
\[
        |c|\le C(A,V).
\]
The equation
\[
        \det M
        =
        e^{-t a\cdot D\psi+tV-c}
\]
then yields
\begin{equation}\label{eq:det-bounds-cell}
        0<\lambda
        \le
        \det M
        \le
        \Lambda<\infty,
\end{equation}
where \(\lambda,\Lambda\) depend only on \(A,a,V\), and \(n\).

\begin{prop}[Uniform second-derivative estimate]
\label{prop:cell-C2-estimate}
There exists a constant
\(C=C(A,a,n,\|V\|_{C^2(\T^n)})\) such that every solution of
\eqref{eq:cell-path} satisfies
\begin{equation}\label{eq:C2-cell}
        C^{-1}I
        \le
        A+D^2\psi
        \le
        CI
        \quad\text{on }\T^n.
\end{equation}
\end{prop}

\begin{proof}
Define the full linearized operator
\[
        \mathcal L_t
        =
        M^{ij}\partial_{ij}
        +t\,a\cdot D.
\]
Fix a constant \(B>0\), to be chosen below, and consider
\[
        G(x,\xi)
        =
        \log M_{\xi\xi}(x)-B\psi(x),
        \quad
        (x,\xi)\in\T^n\times \mathbb S^{n-1}.
\]
Let \((x_0,\xi_0)\) be a maximum point of \(G\).  After a constant orthogonal
change of coordinates, we may assume that at \(x_0\),
\[
        \xi_0=e_1
\]
and
\[
        M(x_0)
        =
        \operatorname{diag}(M_1,\ldots,M_n),
        \quad
        M_1=M_{11}=\lambda_{\max}(M(x_0)).
\]
All computations below are performed at \(x_0\).

Differentiating \eqref{eq:cell-path} twice in the \(x_1\)-direction gives
\[
        M^{ij}M_{ij,11}
        -
        M^{ip}M^{jq}M_{ij,1}M_{pq,1}
        +
        t a_k M_{11,k}
        =
        t V_{11}.
\]
Since
\[
        M_{ij,11}=M_{11,ij},
\]
we obtain
\[
\begin{aligned}
        M^{ij}(\log M_{11})_{ij}
        &=
        \frac{M^{ip}M^{jq}M_{ij,1}M_{pq,1}}{M_{11}} \\
        &\quad
        -
        \frac{M^{ij}M_{11,i}M_{11,j}}{M_{11}^2}
        -
        \frac{t a_k M_{11,k}}{M_{11}}
        +
        \frac{t V_{11}}{M_{11}}.
\end{aligned}
\]
At the diagonal point, the difference of the two third-order quadratic
expressions is nonnegative.  The full symmetry of the third derivatives gives
\[
        M_{11,i}=M_{1i,1}.
\]
Hence
\[
\begin{aligned}
&\frac{M^{ip}M^{jq}M_{ij,1}M_{pq,1}}{M_{11}}
-
\frac{M^{ij}M_{11,i}M_{11,j}}{M_{11}^2} \\
&\qquad =
\sum_{i>1}
\frac{M_{1i,1}^2}{M_1^2M_i}
+
\sum_{i,j>1}
\frac{M_{ij,1}^2}{M_iM_jM_1}
\ge0.
\end{aligned}
\]
Therefore
\[
        M^{ij}(\log M_{11})_{ij}
        \ge
        -\frac{t a_k M_{11,k}}{M_{11}}
        +
        \frac{t V_{11}}{M_{11}}.
\]
The drift part of \(\mathcal L_t\) cancels the third-order drift term:
\[
        t\,a\cdot D\log M_{11}
        =
        \frac{t a_k M_{11,k}}{M_{11}}.
\]
Thus
\begin{equation}\label{eq:cell-log-eigenvalue}
        \mathcal L_t\log M_{11}
        \ge
        \frac{t V_{11}}{M_{11}}.
\end{equation}

On the other hand,
\[
\begin{aligned}
        \mathcal L_t(-B\psi)
        &=
        -B M^{ij}\psi_{ij}
        -Bt\,a\cdot D\psi \\
        &=
        -Bn
        +B\tr(M^{-1}A)
        -Bt\,a\cdot D\psi,
\end{aligned}
\]
because
\[
        M^{ij}\psi_{ij}
        =
        M^{ij}(M_{ij}-A_{ij})
        =
        n-\tr(M^{-1}A).
\]
Since \(G(\,\cdot\,,e_1)\) has a maximum at \(x_0\),
\[
        \mathcal L_t G(x_0,e_1)\le0.
\]
Combining this with \eqref{eq:cell-log-eigenvalue}, we obtain
\begin{equation}\label{eq:cell-C2-core}
        0
        \ge
        B\tr(M^{-1}A)
        -Bn
        -Bt\,a\cdot D\psi
        +
        \frac{t V_{11}}{M_{11}}.
\end{equation}

Let
\[
        \alpha=\lambda_{\min}(A)>0,
        \qquad
        K_V
        =
        \sup_{x\in\T^n}\|D^2V(x)\|_{\mathrm{op}}.
\]
Then
\[
        \tr(M^{-1}A)
        \ge
        \alpha\tr(M^{-1})
        \ge
        \frac{\alpha}{M_{11}}.
\]
Hence
\[
        \frac{t V_{11}}{M_{11}}
        \ge
        -
        \frac{K_V}{\alpha}
        \tr(M^{-1}A).
\]
Using the \(C^1\) bound from
Lemma~\ref{lem:periodic-semiconvexity}, inequality
\eqref{eq:cell-C2-core} gives
\[
        0
        \ge
        \left(
        B-\frac{K_V}{\alpha}
        \right)
        \tr(M^{-1}A)
        -
        B\left(
        n+|a|\|D\psi\|_{L^\infty}
        \right).
\]
Choose
\[
        B
        =
        1+\frac{K_V}{\alpha}.
\]
Then
\[
        \tr(M(x_0)^{-1}A)\le C.
\]
Since \(A^{1/2}M(x_0)^{-1}A^{1/2}\) is positive definite, its operator norm
is bounded by its trace.  Hence
\[
        A^{1/2}M(x_0)^{-1}A^{1/2}\le CI,
\]
and therefore
\[
        M(x_0)\ge C^{-1}A\ge C^{-1}I.
\]
Combining this lower eigenvalue bound with the upper determinant bound in
\eqref{eq:det-bounds-cell}, we obtain
\[
        M_{11}(x_0)\le C.
\]

Since \(G\) attains its maximum at \((x_0,e_1)\), for every
\((x,\xi)\in\T^n\times \mathbb S^{n-1}\),
\[
        \log M_{\xi\xi}(x)-B\psi(x)
        \le
        \log M_{11}(x_0)-B\psi(x_0).
\]
The \(C^0\) estimate for \(\psi\) therefore gives
\[
        M_{\xi\xi}(x)\le C.
\]
Thus
\[
        M\le CI
        \quad\text{on }\T^n.
\]
Finally, the lower determinant bound in
\eqref{eq:det-bounds-cell} and the upper eigenvalue bound imply
\[
        M\ge C^{-1}I.
\]
This proves \eqref{eq:C2-cell}.
\end{proof}

Higher-order estimates follow from standard regularity theory.

\begin{prop}
\label{prop:cell-higher-estimates}
For every integer \(k\ge0\), there exists a constant
\[
        C_k
        =
        C_k\bigl(A,a,n,\|V\|_{C^{k+2}(\T^n)}\bigr)
\]
such that every zero-average solution of \eqref{eq:cell-path} satisfies
\begin{equation}\label{eq:cell-higher-estimates}
        \|\psi\|_{C^k(\T^n)}
        +|c|
        \le C_k.
\end{equation}
\end{prop}

\begin{proof}
Proposition~\ref{prop:cell-C2-estimate} gives uniform ellipticity:
\[
        C^{-1}I
        \le
        A+D^2\psi
        \le
        CI.
\]
Since \(D^2\psi\) is uniformly bounded, \(D\psi\) is uniformly Lipschitz.
Hence the right-hand side of
\[
        \log\det(A+D^2\psi)
        =
        -t a\cdot D\psi+tV-c
\]
is uniformly bounded in \(C^\alpha(\T^n)\) for every fixed
\(\alpha\in(0,1)\).  The Evans--Krylov estimate for concave uniformly
elliptic equations \cite{Evans82,Krylov83} therefore gives
\[
        \|\psi\|_{C^{2,\alpha}(\T^n)}
        \le C.
\]
Differentiating the equation produces uniformly elliptic linear equations for
the derivatives of \(\psi\).  Standard Schauder estimates and induction yield
\eqref{eq:cell-higher-estimates}.
\end{proof}

\subsection{The continuity method}

\begin{proof}[Proof of Theorem~\ref{thm:cell}]
Let \(\mathcal S\subset[0,1]\) be the set of parameters \(t\) for which there
exists a smooth pair \((\psi,c)\) solving
\eqref{eq:cell-path}, with \(A+D^2\psi>0\)
and
\[
        \int_{\T^n}\psi=0.
\]

At \(t=0\), the pair
\[
        \psi=0,
        \quad
        c=-\log\det A
\]
solves \eqref{eq:cell-path}.  Hence
\[
        0\in\mathcal S.
\]

To prove openness, fix \(t_0\in\mathcal S\), and let
\((\psi,c)\) be the corresponding solution.  Set
\[
        M=A+D^2\psi.
\]
The linearization of the left-hand side of
\eqref{eq:cell-path} in the variables \((\psi,c)\) is
\begin{equation}\label{eq:cell-linearization}
        (\eta,\kappa)
        \longmapsto
        L\eta+\kappa,
        \quad
        L\eta
        =
        M^{ij}\eta_{ij}
        +t_0a\cdot D\eta.
\end{equation}
We view this as a map
\[
        C^{2,\alpha}_0(\T^n)\times\R
        \longrightarrow
        C^\alpha(\T^n),
\]
where
\[
        C^{2,\alpha}_0(\T^n)
        =
        \left\{
        \eta\in C^{2,\alpha}(\T^n):
        \int_{\T^n}\eta=0
        \right\}.
\]

The strong maximum principle gives
\[
        \ker L=\{\text{constants}\}.
\]
As an elliptic operator on the compact manifold \(\T^n\), \(L\) is Fredholm
as a map
\[
        C^{2,\alpha}(\T^n)
        \longrightarrow
        C^\alpha(\T^n).
\]
Its index is zero.  Indeed, the lower-order drift can first be removed through
the elliptic homotopy
\[
        L_s\eta=M^{ij}\eta_{ij}+s\,t_0a\cdot D\eta,
        \qquad 0\le s\le1,
\]
and the positive definite principal coefficients can then be deformed through
elliptic operators to those of the flat Laplacian.  The Fredholm index is
constant along these homotopies, while the Laplacian on \(\T^n\) has index
zero.  It follows that
\[
        \operatorname{codim}\operatorname{Ran}L=1.
\]
Also,
\[
        1\notin\operatorname{Ran}L.
\]
Indeed, if \(L\eta=1\), then at a maximum point of \(\eta\), \(L\eta\le0\),
which is impossible.  Hence
\[
        C^\alpha(\T^n)
        =
        \operatorname{Ran}L
        \oplus
        \operatorname{span}\{1\}.
\]
Adding a constant to \(\eta\) does not change \(L\eta\), so restricting the
domain of \(L\) to the zero-average subspace does not change its range.
Hence the map in \eqref{eq:cell-linearization} is surjective.

It is also injective.  If
\[
        L\eta+\kappa=0
\]
and \(\eta\) has zero average, then \(\kappa>0\) is impossible by evaluating
at a minimum point of \(\eta\), while \(\kappa<0\) is impossible by
evaluating at a maximum point.  Thus \(\kappa=0\), and the strong maximum
principle gives that \(\eta\) is constant.  Its zero average then implies
\(\eta=0\).

Thus
\[
        (\eta,\kappa)\longmapsto L\eta+\kappa
\]
is an isomorphism.  The implicit function theorem proves that
\(\mathcal S\) is open.

The \textit{a priori} estimates in
Lemmas~\ref{lem:periodic-semiconvexity} and
Proposition~\ref{prop:cell-higher-estimates} imply that
\(\mathcal S\) is closed.  Indeed, if
\[
        t_j\in\mathcal S,
        \quad
        t_j\rightarrow t_\infty,
\]
then the corresponding solutions admit uniform \(C^k\) bounds for every
\(k\).  A diagonal compactness argument yields a smooth limiting solution at
\(t_\infty\).

Thus \(\mathcal S=[0,1]\).
At \(t=1\), we obtain a smooth pair \((\psi_A,c_A)\) satisfying
\eqref{eq:cell-again}.  Lemma~\ref{lem:cell-uniqueness} gives uniqueness of
\(c_A\), and uniqueness of the zero-average solution \(\psi_A\).
\end{proof}

\section{Periodic correctors}
\label{sec:periodic-correctors-rigidity}

Throughout this section, \(A=A^\top>0\) is fixed, and
\((\psi_A,c_A)\) denotes the unique zero-average solution of the cell problem
\begin{equation}\label{eq:cell-reference-section4}
        \det(A+D^2\psi_A)
        =
        \exp\{-a\cdot D\psi_A+V-c_A\},
        \quad
        A+D^2\psi_A>0
        \quad\text{on }\T^n.
\end{equation}
We extend \(V\) and \(\psi_A\) periodically to \(\R^n\), and write
\[
        M_A=A+D^2\psi_A,
        \quad
        \Phi_A(x)=\frac12 x^\top A x+\psi_A(x).
\]
Thus
\[
        D^2\Phi_A=M_A.
\]

We study entire solutions of
\begin{equation}\label{eq:periodic-full-again}
        \det D^2u
        =
        \exp\{-a\cdot Du+b\cdot x+V(x)-c_0\}
        \quad\text{in }\R^n.
\end{equation}
\subsection{Quadratic-periodic solutions}

\begin{thm}
Let \(A=A^\top>0\), \(\ell\in\R^n\), and
\(\psi\in C^\infty(\T^n)\).  Suppose that
\[
        A+D^2\psi>0.
\]
Then
\[
        u(x)
        =
        \frac12 x^\top A x+\ell\cdot x+\psi(x)
\]
solves \eqref{eq:periodic-full-again} if and only if
\[
        b=Aa,
        \quad
        c_0+a\cdot\ell=c_A,
        \quad
        \psi=\psi_A+C
\]
for some constant \(C\in\R\).

In particular, if \(b=Aa\), then
\[
        u(x)
        =
        \frac12 x^\top A x+\ell\cdot x+\psi_A(x)+C
\]
is a solution precisely when
\(c_0+a\cdot\ell=c_A\), with \(C\in\R\) arbitrary.
\end{thm}

\begin{proof}
For
\[
        u(x)
        =
        \frac12 x^\top A x+\ell\cdot x+\psi(x),
\]
we have
\[
        Du=Ax+\ell+D\psi,
        \quad
        D^2u=A+D^2\psi.
\]
Since \(A\) is symmetric,
\[
\begin{aligned}
        -a\cdot Du+b\cdot x+V-c_0
        &=
        (b-Aa)\cdot x
        -a\cdot\ell
        -a\cdot D\psi
        +V-c_0.
\end{aligned}
\]
Hence the equation becomes
\begin{equation}\label{eq:periodic-ansatz-equation}
        \det(A+D^2\psi)
        =
        \exp\{
        (b-Aa)\cdot x
        -a\cdot D\psi
        +V-(c_0+a\cdot\ell)
        \}.
\end{equation}

The left-hand side of \eqref{eq:periodic-ansatz-equation} is
\(\mathbb Z^n\)-periodic.  The functions \(D\psi\) and \(V\) are also
\(\mathbb Z^n\)-periodic.  Set
\[
        m=b-Aa.
\]
If \(m\neq0\), then some coordinate \(m_i\) is nonzero.  Taking
\(k=e_i\in\mathbb Z^n\) and comparing
\eqref{eq:periodic-ansatz-equation} at \(x\) and \(x+k\) gives
\[
        e^{m\cdot k}=1.
\]
This is impossible because \(m\cdot k=m_i\neq0\).  Therefore
\[
        b=Aa.
\]

With this compatibility imposed, \eqref{eq:periodic-ansatz-equation} reduces
to
\[
        \det(A+D^2\psi)
        =
        \exp\{
        -a\cdot D\psi
        +V-(c_0+a\cdot\ell)
        \}.
\]
The uniqueness of the cell problem, without imposing the additive
normalization on \(\psi\), gives
\[
        c_0+a\cdot\ell=c_A
\]
and
\[
        \psi=\psi_A+C
\]
for some constant \(C\).  Conversely, under these conditions the full-space
equation reduces to \eqref{eq:cell-reference-section4}.  The additive
constant does not occur in
the equation.
\end{proof}

\begin{rem}
If \(a\neq0\) and \(b=Aa\), then for every \(c_0\in\R\) one can choose
\(\ell\in\R^n\) so that
\[
        a\cdot\ell=c_A-c_0.
\]
If \(a=0\), then \(b=Aa\) is equivalent to \(b=0\), and the condition on the
normalizing constant reduces to
\[
        c_0=c_A.
\]
\end{rem}

\subsection{The compatibility condition}

We use the following elementary semiconvexity lemma.

\begin{lem}
\label{lem:semiconvex-gradient-growth}
Let \(f\in C^2(\R^n)\) satisfy
\[
        D^2f\ge-\Lambda I
        \quad\text{in }\R^n
\]
for some \(\Lambda\ge0\).

\begin{enumerate}[label=\textnormal{(\roman*)}]
\item
If \(f\in L^\infty(\R^n)\), then
\begin{equation}\label{eq:bounded-semiconvex-gradient}
        \|Df\|_{L^\infty(\R^n)}
        \le
        \sqrt{2\Lambda\,
        \osc_{\R^n}f}.
\end{equation}
When \(\Lambda=0\), the conclusion is understood as \(Df\equiv0\).

\item
If
\begin{equation}\label{eq:subquadratic-remainder}
        \sup_{B_R}|f|=o(R^2)
        \quad\text{as }R\to\infty,
\end{equation}
then
\begin{equation}\label{eq:sublinear-gradient-growth}
        \sup_{B_R}|Df|=o(R)
        \quad\text{as }R\to\infty.
\end{equation}
\end{enumerate}
\end{lem}

\begin{proof}
Fix \(x\in\R^n\) and \(e\in \mathbb S^{n-1}\), and define
\[
        g(t)=f(x+te).
\]
Then
\[
        g''(t)\ge-\Lambda.
\]

Suppose first that \(f\) is bounded and
\[
        P=g'(0)>0.
\]
For \(0\le t\le P/\Lambda\), when \(\Lambda>0\),
\[
        g'(t)\ge P-\Lambda t.
\]
Therefore
\[
\begin{aligned}
        \osc_{\R^n}f
        &\ge
        g(P/\Lambda)-g(0) \\
        &\ge
        \int_0^{P/\Lambda}
        (P-\Lambda t)\,\mathrm{d}t
        =
        \frac{P^2}{2\Lambda}.
\end{aligned}
\]
Thus
\[
        P\le
        \sqrt{2\Lambda\,\osc_{\R^n}f}.
\]
Applying the same argument in the direction \(-e\) gives the corresponding
lower bound for \(g'(0)\).  This proves
\eqref{eq:bounded-semiconvex-gradient}.  If \(\Lambda=0\), then \(f\) is
bounded and convex, hence constant.

Assume next that \eqref{eq:subquadratic-remainder} holds.  Fix
\(\delta\in(0,1)\).  If \(x\in B_R\), \(e\in \mathbb S^{n-1}\), and
\(t=\delta R\), semiconvexity gives
\[
        f(x+te)
        \ge
        f(x)+tDf(x)\cdot e-\frac{\Lambda t^2}{2}
\]
and
\[
        f(x-te)
        \ge
        f(x)-tDf(x)\cdot e-\frac{\Lambda t^2}{2}.
\]
Since
\[
        x\pm te\in B_{(1+\delta)R},
\]
we obtain
\[
        |Df(x)\cdot e|
        \le
        \frac{
        2\sup_{B_{(1+\delta)R}}|f|
        }{\delta R}
        +
        \frac{\Lambda\delta R}{2}.
\]
Taking the supremum over \(x\in B_R\) and \(e\in \mathbb S^{n-1}\), dividing by
\(R\), and letting \(R\to\infty\), we find
\[
        \limsup_{R\to\infty}
        \frac{\sup_{B_R}|Df|}{R}
        \le
        \frac{\Lambda\delta}{2}.
\]
Since \(\delta>0\) is arbitrary, this proves
\eqref{eq:sublinear-gradient-growth}.
\end{proof}

\begin{prop}
\label{prop:asymptotic-compatibility}
Let
\[
        u(x)
        =
        \frac12 x^\top A x+\ell\cdot x+f(x),
        \quad
        A=A^\top>0,
\]
be a smooth strictly convex solution of
\eqref{eq:periodic-full-again}.  Assume that
\[
        \omega(R):=
        \sup_{B_R}|Df|
        =
        o(R)
        \quad\text{as }R\to\infty.
\]
Then
\[
        b=Aa.
\]
\end{prop}

\begin{proof}
Set
\[
        m=b-Aa.
\]
Since
\[
        Du=Ax+\ell+Df,
\]
the equation becomes
\[
        \det D^2u
        =
        \exp\{
        m\cdot x
        -a\cdot\ell
        -a\cdot Df
        +V(x)-c_0
        \}.
\]

Suppose, for contradiction, that \(m\neq0\), and set
\[
        e=\frac{m}{|m|}.
\]
For \(R>1\), define the spherical cap
\[
        E_R
        =
        \left\{
        x\in B_R:
        e\cdot x\ge\frac R2
        \right\}.
\]
There exists \(c_n>0\) such that
\[
        |E_R|\ge c_n R^n.
\]
On \(E_R\),
\[
        m\cdot x\ge\frac{|m|R}{2}.
\]
Since \(V\) is bounded and
\[
        |a\cdot Df|
        \le
        |a|\omega(R)
        =
        o(R),
\]
for all sufficiently large \(R\) we have, on \(E_R\),
\[
        \log\det D^2u
        \ge
        \frac{|m|R}{4}-C.
\]
Hence
\begin{equation}\label{eq:compat-exponential-lower}
        \int_{B_R}\det D^2u\,\mathrm{d}x
        \ge
        cR^n e^{|m|R/4}.
\end{equation}

Because \(D^2u>0\), the gradient map \(Du\) is injective.  Indeed,
\[
        (Du(x)-Du(y))\cdot(x-y)>0
        \quad
        \text{for }x\neq y.
\]
Since \(Du\) is injective, the area formula gives
\begin{equation}\label{eq:gradient-area-formula}
        \int_{B_R}\det D^2u\,\mathrm{d}x
        =
        |Du(B_R)|.
\end{equation}
On the other hand,
\[
        Du(B_R)
        \subset
        AB_R+\ell+B_{\omega(R)}.
\]
Since \(\omega(R)=o(R)\), the set on the right is contained in a Euclidean
ball of radius \(CR\) for all sufficiently large \(R\).  Hence
\begin{equation}\label{eq:compat-polynomial-upper}
        |Du(B_R)|\le CR^n.
\end{equation}
Together with \eqref{eq:gradient-area-formula}, the bounds
\eqref{eq:compat-exponential-lower} and
\eqref{eq:compat-polynomial-upper} give a contradiction.  Hence \(m=0\),
that is, \(b=Aa\).
\end{proof}

\begin{cor}
\label{cor:subquadratic-compatibility}
Let
\[
        u(x)
        =
        \frac12 x^\top A x+\ell\cdot x+f(x),
        \quad
        A=A^\top>0,
\]
be a smooth strictly convex solution of
\eqref{eq:periodic-full-again}.

If
\[
        \sup_{B_R}|f|=o(R^2)
        \quad\text{as }R\to\infty,
\]
then
\[
        b=Aa.
\]
\end{cor}

\begin{proof}
The strict convexity condition gives
\[
        A+D^2f>0,
\]
and hence
\[
        D^2f\ge-\lambda_{\max}(A)I.
\]
Lemma~\ref{lem:semiconvex-gradient-growth} gives
\[
        \sup_{B_R}|Df|=o(R),
\]
so Proposition~\ref{prop:asymptotic-compatibility} applies.
\end{proof}

\subsection{Bounded correctors}

We first prove a compactness lemma for integer translates.

\begin{lem}
\label{lem:translation-compactness}
Let
\[
        h\in C^\infty(\R^n)\cap L^\infty(\R^n)
\]
satisfy
\[
        M_A+D^2h>0
\]
and
\begin{equation}\label{eq:h-delta}
        \log\det(M_A+D^2h)
        -
        \log\det M_A
        +
        a\cdot Dh
        =
        \delta
        \quad\text{in }\R^n
\end{equation}
for some constant \(\delta\in\R\).  Then the family
\[
        \left\{
        h(\,\cdot\,+m):m\in\mathbb Z^n
        \right\}
\]
is precompact in \(C^\infty_{\mathrm{loc}}(\R^n)\).
\end{lem}

\begin{proof}
Since \(M_A\) is smooth, periodic, and positive definite, there are constants
\[
        0<\lambda_A\le\Lambda_A<\infty
\]
such that
\begin{equation}\label{eq:MA-uniform-bounds}
        \lambda_A I\le M_A\le\Lambda_A I
        \quad\text{on }\R^n.
\end{equation}
The positivity of \(M_A+D^2h\) gives
\[
        D^2h\ge-\Lambda_A I.
\]
Lemma~\ref{lem:semiconvex-gradient-growth} and the boundedness of \(h\) then
yield
\begin{equation}\label{eq:h-global-gradient-bound}
        \|Dh\|_{L^\infty(\R^n)}\le L
\end{equation}
for some finite constant \(L\).  Set also
\[
        H=\|h\|_{L^\infty(\R^n)}.
\]

Equation \eqref{eq:h-delta} is equivalent to
\[
        \det(M_A+D^2h)
        =
        e^{\delta-a\cdot Dh}\det M_A.
\]
Using \eqref{eq:MA-uniform-bounds} and
\eqref{eq:h-global-gradient-bound}, we obtain constants
\[
        0<\lambda\le\Lambda<\infty
\]
such that
\begin{equation}\label{eq:h-determinant-bounds}
        \lambda
        \le
        \det(M_A+D^2h)
        \le
        \Lambda
        \quad\text{in }\R^n.
\end{equation}

For \(m\in\mathbb Z^n\), set
\[
        h_m(x)=h(x+m),
        \qquad
        U_m(x)=\Phi_A(x)+h_m(x).
\]
The periodicity of \(M_A\) implies
\[
        D^2U_m=M_A+D^2h_m>0,
\]
and every \(h_m\) satisfies
\begin{equation}\label{eq:translated-h-equation}
        \log\det(M_A+D^2h_m)
        -
        \log\det M_A
        +
        a\cdot Dh_m
        =
        \delta.
\end{equation}
Equivalently,
\begin{equation}\label{eq:translated-MA-equation}
        \det D^2U_m
        =
        e^{\delta-a\cdot Dh_m}\det M_A.
\end{equation}
The determinant bounds in \eqref{eq:h-determinant-bounds} hold uniformly for
all \(m\).

Fix
\(x_0\in\R^n\), and define
\[
\begin{aligned}
        q_{m,x_0}(z)
        &:=
        U_m(x_0+z)-U_m(x_0)-DU_m(x_0)\cdot z,\\
        S_m(x_0,t)
        &:=
        \left\{
        x_0+z:q_{m,x_0}(z)<t
        \right\}.
\end{aligned}
\]
By \eqref{eq:MA-uniform-bounds},
\[
        \frac{\lambda_A}{2}|z|^2
        \le
        \Phi_A(x_0+z)-\Phi_A(x_0)-D\Phi_A(x_0)\cdot z
        \le
        \frac{\Lambda_A}{2}|z|^2.
\]
By the definitions of \(H\) and \(L\),
\[
\left|
        h_m(x_0+z)-h_m(x_0)-Dh_m(x_0)\cdot z
\right|
        \le
        2H+L|z|.
\]
Hence
\begin{equation}\label{eq:uniform-section-growth}
        \frac{\lambda_A}{2}|z|^2-L|z|-2H
        \le
        q_{m,x_0}(z)
        \le
        \frac{\Lambda_A}{2}|z|^2+L|z|+2H.
\end{equation}
All constants here are independent of \(m\) and \(x_0\).

Choose
\[
        H_0
        =
        4\bigl(2\Lambda_A+2L+2H+1\bigr).
\]
The upper bound in \eqref{eq:uniform-section-growth} gives
\begin{equation}\label{eq:section-inner-ball}
        B_2(x_0)
        \subset
        S_m\left(x_0,\frac{H_0}{4}\right).
\end{equation}
Choose \(R_0>2\) so large that
\[
        \frac{\lambda_A}{2}r^2-Lr-2H>H_0
        \quad\text{for every }r\ge R_0.
\]
The lower bound in \eqref{eq:uniform-section-growth} then gives
\begin{equation}\label{eq:section-outer-ball}
        S_m(x_0,H_0)
        \subset
        B_{R_0}(x_0).
\end{equation}
Hence the sections \(S_m(x_0,H_0)\) have uniform inner and outer ball
bounds.  Set
\[
        \Omega_{m,x_0}
        =
        S_m(x_0,H_0)-x_0.
\]
By John's lemma, there is an affine map
\[
        \mathcal A_{m,x_0}(z)
        =
        B_{m,x_0}z+y_{m,x_0}
\]
such that
\[
        B_1
        \subset
        \mathcal A_{m,x_0}(\Omega_{m,x_0})
        \subset
        B_n.
\]
The ball inclusions \eqref{eq:section-inner-ball} and
\eqref{eq:section-outer-ball} imply uniform bounds for
\(B_{m,x_0}\), \(B_{m,x_0}^{-1}\), and
\(|\det B_{m,x_0}|^{\pm1}\).  Define on the normalized domain
\[
        \widetilde\Omega_{m,x_0}
        =
        \mathcal A_{m,x_0}(\Omega_{m,x_0})
\]
the convex function
\[
        \widetilde v_{m,x_0}(y)
        =
        \frac{1}{H_0}
        q_{m,x_0}\bigl(
        \mathcal A_{m,x_0}^{-1}(y)
        \bigr)
        -1.
\]
Then \(\widetilde v_{m,x_0}=0\) on
\(\partial\widetilde\Omega_{m,x_0}\).  Since
\(q_{m,x_0}\ge0\) and \(q_{m,x_0}(0)=0\),
\[
        \min_{\widetilde\Omega_{m,x_0}}
        \widetilde v_{m,x_0}
        =
        \widetilde v_{m,x_0}
        \bigl(\mathcal A_{m,x_0}(0)\bigr)
        =
        -1.
\]
The normalized domains and potentials therefore satisfy the standard
normalization hypotheses uniformly in \(m\) and \(x_0\).  Moreover,
\[
\begin{aligned}
        \det D^2\widetilde v_{m,x_0}(y)
        & =
        H_0^{-n}|\det B_{m,x_0}|^{-2}
        \det D^2U_m\bigl(
        x_0+\mathcal A_{m,x_0}^{-1}(y)
        \bigr),
\end{aligned}
\]
so the normalized Monge--Amp\`ere densities are bounded above and below by
positive constants independent of \(m\) and \(x_0\).

By Caffarelli's interior \(C^{1,\gamma}\) estimate on normalized sections
\cite{CaffarelliLocalization,CaffarelliRegularity91}
(see also \cite[Theorem~4.20]{Figalli17}), there exist
\(\gamma\in(0,1)\) and a uniform constant \(C\) such that
\[
        [DU_m]_{C^\gamma(S_m(x_0,3H_0/4))}
        \le C.
\]
All these estimates are applied on fixed fractional sublevel sets of the
normalized potential, so the constants are uniform in \(m\) and \(x_0\).
Since
\[
        Dh_m=DU_m-D\Phi_A
\]
and \(\Phi_A\) is fixed and smooth, the right-hand side
\[
        f_m(x)
        =
        e^{\delta-a\cdot Dh_m(x)}\det M_A(x)
\]
of \eqref{eq:translated-MA-equation} is uniformly bounded in
\(C^\gamma(S_m(x_0,3H_0/4))\) and is uniformly bounded away from zero.
After pulling the equation back by \(\mathcal A_{m,x_0}\), the normalized
right-hand sides also have uniformly bounded \(C^\gamma\) norms, since the
linear parts of \(\mathcal A_{m,x_0}\) and their inverses are uniformly
bounded.

Caffarelli's interior \(C^{2,\alpha}\) estimate for H\"older densities
\cite{CaffarelliW2p,CaffarelliRegularity91}, in the formulation
\cite[Theorem~4.42]{Figalli17}, applied to the normalized equation and then
transferred back to the original coordinates, yields some
\(\alpha\in(0,\gamma]\) such that
\[
        \|D^2U_m\|_{C^\alpha(S_m(x_0,H_0/2))}
        \le C.
\]
The affine changes of variables above have uniformly bounded linear parts and
inverses, so the constants remain independent of \(m\) and \(x_0\).  Since
\[
        B_2(x_0)
        \subset
        S_m\left(x_0,\frac{H_0}{4}\right)
        \subset
        S_m\left(x_0,\frac{H_0}{2}\right),
\]
we obtain a uniform \(C^{2,\alpha}\) bound on \(B_2(x_0)\).  Together with
the determinant lower bound in \eqref{eq:h-determinant-bounds}, this makes
\[
        G_m
        :=
        D^2U_m
        =
        M_A+D^2h_m
\]
uniformly elliptic on \(B_2(x_0)\).

Differentiating \eqref{eq:translated-h-equation} in the \(x_k\)-direction gives
\begin{equation}\label{eq:translated-first-derivative}
        G_m^{ij}\partial_{ij}(h_m)_k
        +
        a\cdot D(h_m)_k
        =
        \left(
        (M_A^{-1})^{ij}-G_m^{ij}
        \right)
        \partial_k(M_A)_{ij}.
\end{equation}
The coefficients \(G_m^{ij}\) are uniformly elliptic and uniformly
\(C^\alpha\), and the right-hand side is uniformly \(C^\alpha\), because
\(M_A\) is fixed and smooth.  Interior Schauder estimates applied to
\eqref{eq:translated-first-derivative} give a uniform
\(C^{3,\alpha}\) bound for \(h_m\) on \(B_{3/2}(x_0)\).  Repeated
differentiation and induction then yield, for every integer \(r\ge0\),
\[
        \|h_m\|_{C^r(B_1(x_0))}
        \le
        C_r,
\]
where \(C_r\) depends only on \(r\), the fixed data, and finitely many
norms of the fixed smooth matrix field \(M_A\), but not on \(m\) or
\(x_0\).  Hence the family \(\{h_m\}\) is uniformly bounded in every
\(C^r\) norm on compact subsets of \(\R^n\).  The Arzel\`a--Ascoli theorem
and a diagonal argument give precompactness in
\(C^\infty_{\mathrm{loc}}(\R^n)\).
\end{proof}

\begin{thm}
\label{thm:h-liouville}
Let
\(h\in C^\infty(\R^n)\cap L^\infty(\R^n)\)
satisfy
\[
        M_A+D^2h>0
\]
and
\[
        \log\det(M_A+D^2h)
        -
        \log\det M_A
        +
        a\cdot Dh
        =
        \delta
        \quad\text{in }\R^n
\]
for some constant \(\delta\in\R\).  Then
\[
        \delta=0
\quad\text{and}\quad
        h=\textnormal{constant}.
\]
\end{thm}

\begin{proof}
Fix \(k\in\mathbb Z^n\), and define
\[
        C_k
        =
        \sup_{x\in\R^n}
        \bigl(h(x+k)-h(x)\bigr).
\]
Since \(h\) is bounded, \(C_k\) is finite.  Choose a sequence
\(x_j\in\R^n\) such that
\[
        h(x_j+k)-h(x_j)\rightarrow C_k.
\]
Write
\[
        x_j=m_j+y_j,
        \quad
        m_j\in\mathbb Z^n,
        \quad
        y_j\in[0,1]^n.
\]
After passing to a subsequence,
\[
        y_j\rightarrow y_\infty
        \in[0,1]^n.
\]
By Lemma~\ref{lem:translation-compactness}, after passing to a further
subsequence,
\[
        h(\,\cdot\,+m_j)
        \rightarrow H
        \quad\text{in }C^\infty_{\mathrm{loc}}(\R^n),
\]
where \(H\) is bounded and satisfies the same equation
\[
        \log\det(M_A+D^2H)
        -
        \log\det M_A
        +
        a\cdot DH
        =
        \delta.
\]

For every \(x\in\R^n\),
\[
        h(x+m_j+k)-h(x+m_j)\le C_k.
\]
Passing to the limit gives
\begin{equation}\label{eq:limit-translation-upper}
        H(x+k)-H(x)\le C_k
        \quad\text{for all }x\in\R^n.
\end{equation}
The choice of the maximizing sequence gives equality at \(y_\infty\):
\begin{equation}\label{eq:limit-translation-touch}
        H(y_\infty+k)-H(y_\infty)=C_k.
\end{equation}

Define
\[
        \widetilde H(x)=H(x+k)-C_k.
\]
By \eqref{eq:limit-translation-upper},
\[
        \widetilde H\le H,
\]
and by \eqref{eq:limit-translation-touch}, equality holds at
\(y_\infty\).  Since \(k\in\mathbb Z^n\) and \(M_A\) is periodic, the
function \(H(\,\cdot\,+k)\) satisfies the same equation as \(H\).  Subtracting
the constant \(C_k\) also leaves the equation unchanged.  Thus both
\(\widetilde H\) and \(H\) solve
\[
        \log\det(M_A+D^2v)+a\cdot Dv
        =
        \log\det M_A+\delta.
\]
The strong comparison principle,
Lemma~\ref{lem:strong-comparison}, yields
\[
        H(x+k)-C_k\equiv H(x).
\]
Iterating,
\[
        H(x+qk)-H(x)=qC_k
        \quad
        \text{for every }q\in\mathbb N.
\]
Since \(H\) is bounded, this is possible only if
\(C_k=0\).
Therefore
\[
        h(x+k)-h(x)\le0
        \quad\text{for all }x.
\]
Applying the same argument to \(-k\) and then replacing \(x\) by \(x+k\)
gives
\[
        h(x)-h(x+k)\le0.
\]
Hence
\[
        h(x+k)=h(x)
        \quad
        \text{for all }x\in\R^n,\quad k\in\mathbb Z^n.
\]
Thus \(h\) is \(\mathbb Z^n\)-periodic.

At a maximum point \(x_+\) of \(h\),
\[
        Dh(x_+)=0,
        \quad
        D^2h(x_+)\le0.
\]
Therefore
\[
        M_A(x_+)+D^2h(x_+)
        \le
        M_A(x_+),
\]
and the equation gives
\(\delta\le0\).
At a minimum point \(x_-\),
\[
        Dh(x_-)=0,
        \quad
        D^2h(x_-)\ge0,
\]
so the same equation gives
\(\delta\ge0\).
Hence \(\delta=0\), and \(\psi_A+h\) and \(\psi_A\) solve the same
periodic cell equation:
\[
        \log\det(A+D^2v)+a\cdot Dv=V-c_A.
\]
The uniqueness statement from Section~\ref{sec:periodic-cell} implies that
their difference \(h\) is constant.
\end{proof}

We finally prove Theorem~\ref{thm:bounded-classification}.

\begin{proof}[Proof of Theorem~\ref{thm:bounded-classification}]
Write
\[
        f(x)
        =
        u(x)
        -
        \frac12 x^\top A x
        -
        \ell\cdot x.
\]
The compatibility conclusion under the subquadratic hypothesis follows from
Corollary~\ref{cor:subquadratic-compatibility}.  Assume now that
\(f\in L^\infty(\R^n)\).  Since a bounded function is subquadratic, we
already know that \(b=Aa\).  Also,
\[
        f\in C^\infty(\R^n)\cap L^\infty(\R^n),
\]
and
\[
        A+D^2f=D^2u>0.
\]

Set
\[
        h=f-\psi_A.
\]
Since \(f\) is bounded and \(\psi_A\) is periodic,
\(h\in C^\infty(\R^n)\cap L^\infty(\R^n)\).  Also,
\[
        M_A+D^2h
        =
        A+D^2f
        =
        D^2u
        >0.
\]

Using \(b=Aa\), the equation for \(u\) becomes
\[
        \log\det(A+D^2f)
        =
        -a\cdot Df
        +V-(c_0+a\cdot\ell).
\]
The cell equation gives
\[
        \log\det M_A
        =
        -a\cdot D\psi_A+V-c_A.
\]
Subtracting these identities and using
\(Dh=Df-D\psi_A\),
we obtain
\[
        \log\det(M_A+D^2h)
        -
        \log\det M_A
        +
        a\cdot Dh
        =
        c_A-(c_0+a\cdot\ell).
\]
Theorem~\ref{thm:h-liouville} therefore implies
\[
        c_A-(c_0+a\cdot\ell)=0
\quad\text{and}\quad
        h=\text{constant}.
\]
Thus
\[
        c_0+a\cdot\ell=c_A
\quad\text{and}\quad
        f=\psi_A+\text{constant}.
\]
Equivalently,
\[
        u(x)
        =
        \frac12 x^\top A x
        +
        \ell\cdot x
        +
        \psi_A(x)
        +
        C
\]
for some \(C\in\R\).
\end{proof}

\section*{Declaration of competing interest}
The author declares that there are no known competing financial interests or
personal relationships that could have appeared to influence the work reported
in this paper.

\section*{Declaration of generative AI and AI-assisted technologies in the
manuscript preparation process}
During the preparation of this work, the author used OpenAI's ChatGPT to assist
with literature organization and with improving the language, clarity, and
structure of the manuscript.  After using this tool, the author reviewed and
edited the content as needed, independently verified the mathematical arguments
and bibliographic information, and takes full responsibility for the content of
the publication.

\end{document}